\documentclass[12pt, 14paper,reqno]{amsart}
\usepackage{amsmath,amsfonts,amssymb}
\usepackage{mathrsfs}
\usepackage{longtable}
\usepackage[breaklinks]{hyperref}
\usepackage{graphicx}
\usepackage{setspace}
\makeatletter
\@namedef{subjclassname@2020}{%
  \textup{2020} Mathematics Subject Classification}
\makeatother
\theoremstyle{plain}
\newtheorem{thm}{Theorem}[section]

\theoremstyle{proof}
\numberwithin{equation}{section}

\newcommand \QQ {{\mathbb Q}}
\newcommand \ZZ {{\mathbb Z}}

\input xypic
\xyoption{all}
\usepackage[OT2,T1]{fontenc}
\DeclareSymbolFont{cyrletters}{OT2}{wncyr}{m}{n}
\DeclareMathSymbol{\Sha}{\mathalpha}{cyrletters}{"58}

\begin{document}
\title{Strong BSD for abelian surfaces and the Bloch-Beilinson conjecture}
\author{Kalyan Banerjee}
\email{kalyan.ba@srmap.edu.in}
\maketitle

\begin{abstract}
In this paper, we prove the Bloch-Beilinson conjecture for certain abelian surfaces over $ \QQ$, provided that the BSD is known for these abelian surfaces.
\end{abstract}

\section{Introduction}
One of the most important problems in algebraic geometry is to compute the Chow groups for higher codimensional cycles on a smooth projective variety. The theorem due to Mumford in \cite{M}, proves that the Chow group of zero cycles on a smooth projective surface over complex numbers is infinite-dimensional, provided that the geometric genus of the surface is greater than zero, in the sense that the natural maps from symmetric powers to the Chow group are never surjective. The converse is a conjecture due to Bloch, that is, if we have a smooth projective complex surface of geometric genus zero, then the Chow group of zero cycles is isomorphic to the group of points on an abelian variety.

This theorem has been proved for surfaces not of general type with geometric genus zero due to \cite{BKL}. The case for surfaces of general type, that is, surfaces of Kodaira dimension 2 with geometric genus zero, is still open. Some examples of such surfaces for which the Bloch's conjecture holds are due to \cite{B}, \cite{Ba}, \cite{IM}, \cite{PW},\cite{V},\cite{VC}.

The situation in $\bar{\QQ}$ is completely different. The conjecture of Bloch and Beilinson, \cite{Be}, \cite{Bl},  predicts that the group of zero cycles of degree zero denoted by $A_0(S)$ for a smooth projective surface $S$ defined over a number field is finitely generated. This conjecture has not been verified for any examples of smooth, projective surfaces defined over a number field. The aim of this text is to prove that Bloch-Beilinson's conjecture holds for the abelian surfaces defined over a number field.

Let $A$ be an abelian surface defined over a number field $K$.

The main theorems are as follows:

\begin{thm}\label{thm1}
    The group $A_0(A_K)\otimes \QQ$ is finitely generated.
\end{thm}

This is not the full strength of the Bloch-Beilinson conjectures, it relates the rank of the above group with the order of the zero of the L function $L(H^3(A),s)$ at $s=2$, here $H^3(A)$ is the l-adic cohomology of $A$. The conjecture says that,

$$rank(A_0(A_K))=ord_{s=2}L(H^3(A),s)\;.$$

There is also the famous Birch-Swinnerton Dyer conjecture for abelian surfaces over $\QQ$, which says that for any abelian surface $A$ defied over $\QQ$, we have

$$ord_{s=1}L(A,s)=rank(A(\QQ))\;.$$

In particular, it predicts that, if $L(A,1)\neq 0$, then the rank of $A$ over $\QQ$ is zero. The next theorem, which follows from the previous \ref{thm1} is that 

\begin{thm}
We have 
$$rank(A_0(A_{\QQ}))=ord_{s=2}L(H^3(A),s)=0$$
provided that $L(A,1)\neq 0$ and the Birch-Swinnerton Dyer conjecture is known for the abelian surface $A$ defined over $\QQ$.
\end{thm}

Examples of such modular abelian surfaces that are obtained as the isogeny quotient of the Jacobian of a genus two  modular curves $X_0(N)$ of level $N$, are studied in \cite{KS}, \cite{KS1}. In this regard the main theorem of \cite{LZ}[Theorem A] is also important and useful.

Notations: We work over the number fields and their algebraic closures.

$K$ denote a number field and $\bar K$ its algebraic closure.

$A$ denote an abelian surface defined over $K$.

$A_0(A)$ is the group of zero cycles of degree zero modulo rational equivalence over $\bar K$.

$Gal(\bar K/K)$ denote the absolute Galois group.

$A_0(A)(K)$ is the Galois invariant cycles of the group $A_0(A)$ under the action of the Galois group.

$A_0(A_K)$ denote the group of zero cycles of degree zero generated by the $K$-rational points on $A$. 

$T(A)$ the albanese kernel over $\bar K$.

$T(A)(K)$ the Galois invariants in the albanese kernel.

$T(A_K)$ The albanese kernel defined over $K$.

$X$ denote the Kummer's K3 surface.

\section{Preliminaries}

Let $K$ be a number field and let $\overline{K}$ be its algebraic closure. Let $A$ be an abelian variety defined over the number field $K$. Then we have a natural Galois action of the absolute Galois group $Gal(\bar K/K)$ and this action further induces an action on the Chow group of zero cycles on the set of $\bar K$ points of the abelian variety $A$, that is, the free abelian group generated by closed points on $A(\bar K)$ modulo the rational equivalence. Indeed, the action of the Galois group preserves the rational equivalence as given any element of the Galois group, it induces an automorphism of the variety and hence we can consider the pull back of the automorphism at the level of zero cycles and we know that pull back of morphisms preserves rational equivalence. We denote this group by $CH_0(A)$ instead of writing $CH_0(A(\bar K))$. Let $A_0(A)$ be the subgroup of $CH_0(A)$ consisting of zero-cycle classes of degree zero modulo rational equivalence.

Considering the Galois action on $CH_0(A)$ we have the first Galois cohomology $H^1(G,CH_0(A))$ of $CH_0(A)$, for the definition of Galois cohomology, see \cite{Sil}.

\subsection{Tate-Shafarevich and Selmer group of $A_0(A)$ and their properties}
Let $K$ be a number field. We have the following exact sequence as in \cite{Sil}[VII.2] where we replace an elliptic curve $E$ by $A_0(A)$:
$$
0\to A_0(A)(K)/nA_0(A)(K)\to H^1(G,A_0(A)[n])\to H^1(G,A_0(A))[n]\to 0\;.
$$
Let  $v$ be a place of $K$ and $K_v$ be the corresponding completion.
% of $K$ at this place $v$.
 
In short, we write the groups $Gal(\bar K/K)$ and $Gal(\bar K_v/K_v)$ as $G$ and $G_v$, respectively.
Then consider the algebraic closure $\bar K_v$ of $K_v$ and embed $\bar K$ into $\bar K_v$. This embedding gives an injection of the Galois group $G_v$ into $G$. Taking into account the Galois cohomology, we have a homomorphism,
$$
H^1(G,A_0(A(\bar K))\to H^1(G_v,A_0(A(\bar K))\;.
$$

We have the following commutative diagrams:

$$
  \diagram
  A_0(A)(K)/nA_0(A)(K)\ar[dd]_-{} \ar[rr]^-{} & & H^1(G,A_0(A)[n]) \ar[dd]^-{} \\ \\
  A_0(A)(K_v)/nA_0(A)(K_v) \ar[rr]^-{} & & H^1(G_v,A_0(A)[n])
  \enddiagram
  $$

$$
  \diagram
  H^1(G,A_0(A)[n])\ar[dd]_-{} \ar[rr]^-{} & & H^1(G,A_0(A))[n] \ar[dd]^-{} \\ \\
 H^1(G_v,A_0(A)[n])\ar[rr]^-{} & & H^1(G_v,A_0(A))[n]
  \enddiagram
  $$
The above diagrams are commutative following the long exact sequence in Galois cohomology for $G, G_v$ and the fact $G_v$ is a subgroup of $G$ which defines a contravariant homomorphism at the level of long exact sequences.

We now work with the composition of the upper horizontal map and the right vertical map,

$$
H^1(G,A_0(A)[n])\to \prod_v H^1(G_v,A_0(A))[n].
$$

The kernel of this map above is defined to be the $n$ -Selmer group and is denoted by $S^n(A_0(A)/K)$. The above product is over all finite places $v$ in $K$.

The Tate-Shafarevich group is the kernel of the map,
$$H^1(G,A_0(A))\to \prod_v H^1(G_v,A_0(A))\;,$$
and it is denoted by $\Sha(A_0(A)/K)$.

Further we consider the commutative diagram which is induced by the albanese map $A_0(A)\to A$ at the level of cohomology:
$$
  \diagram
  \label{diag1}
  H^1(G,A_0(A))\ar[dd]_-{} \ar[rr]^-{} & & \prod_v H^1(G_v,A_0(A)) \ar[dd]^-{} \\ \\
 H^1(G,A)\ar[rr]^-{} & & \prod_v H^1(G_v,A)
  \enddiagram
$$

We cite the following theorem, which was proved for abelian varieties over global function fields in \cite{BC} and the same proof works over a number field.

\begin{thm}\label{thm1}\cite{BC}[Theorem 1.1]
The group $S^n(A_0(A)/K)$ is embedded in the unramified cohomology group $H^1(G,A_0(A)[n],S)$, where $S$ is a finitely many set of places in $K$ and is finite and therefore $$ A_0 (A) (K) / nA_0 (A) (K)$$ is finite for a number field $K$.
\end{thm}

\section{The finite generation of $A_0(A_K)\otimes \QQ$ for abelian surfaces}

Let $A$ be an abelian surface defined over a number field $K$. We need to prove that

\begin{thm}
    The group $A_0(A_K)\otimes \QQ$ is finitely generated.
\end{thm}

\begin{proof}
For this, we prove that over $\bar K$ the group $A_0(A(\bar K))$ is isomorphic to $A(\bar K)$. We need to prove that the albanese kernel of the natural albanese map 
$$alb_A: A_0(A(\bar K))\to A(\bar K)$$
is trivial. Suppose that there is an albanese kernel $T(A)$ associated with the albanese map that is not trivial. This kernel is defined over some finite extension $L$ of $K$, that is, there exists a non-trivial zero cycle in the albanese kernel which is defined over $L$. Then we have the exact sequence given by :

$$0\to T(A)\to A_0(A(\bar K))\to A(\bar K)\to 0$$

Considering the above sequence at the level of $n$-torsions
for $n\geq 2$ we have 
$$ T (A) [n]\cong 0$$
according to Roitman's torsion theorem. Then we can write the exact sequence 

$$0\to T(A)[n]\to T(A)\stackrel{n}{\to}T(A)\to 0$$

Hence by Roitman's theorem we have $T(A)\to T(A)$ given by $a\to na$ is an isomorphism. Considering the sequence:

$$0\to T(A)(K)/nT(A)(K)\to Sel^n(T(A)/K)\to \Sha(T(A)/K)[n]\to 0 $$
Since $T(A)[n]=0$ we have 

$$Sel^n(T(A)/K)=\prod_{v \textit{finite}} \ker(H^1(G,T(A)[n])\to H^1(G_v, T(A))[n])=0$$

This is due to the fact $T(A)[n]=0$. Then we have 

$$T(A)(K)/nT(A)(K)=0$$ 

Consequently, considering the sequence,

$$0\to T(A)(K)/nT(A)(K)\to A_0(A)(K)/nA_0(A)(K)\to A(K)/nA(K)$$

we see that $$A_0(A)(K)/nA_0(A)(K)\to A(K)/nA(K) $$
is injective. 

Now we know that by Mordell-Weill theorem $A(K)$ is finitely generated. Here we have the commutative diagram,
$$\diagram
   A_0(A)(K)\ar[dd] \ar[rr] & & A(K)\ar[dd] \\ \\
  A_0(A)(K)/nA_0(A)(K) \ar[rr] & & A(K)/nA(K)\\
  \enddiagram
  $$
Since the group $A(K)$ is finitely generated if $A_0(A)(K)$ is not finitely generated, there is a kernel $T(A)(K)$, such that 

$$A_0(A)(K)/T(A)(K)\subset A(K)$$

is finitely generated, and $T(A)(K)$ is not finitely generated. At the same time, we also have

$$T(A)(K)=nT(A)(K)\;.$$

Take an element $\alpha\neq 0$ in $$A_0(A)(K)/nA_0(A)(K)$$ and lift it to an element $\widetilde{\alpha}$ of $A_0(A)(K)$. Then $\widetilde{\alpha}$ is not equal to $n\beta$ for any other $\beta$ in $A_0(A)(K)$. Suppose that there are infinitely many distinct elements $\widetilde{\alpha_i}$ that are $\ZZ$-linearly independent and correspond to the distinct elements $\alpha_i$ of $A_0(A)(K)/nA_0(A)(K)$. We show that this is not possible. If possible, then, there are infinitely many such $\ZZ$-linearly independent elements in $T(A)(K)$ and we have  

$$T(A)(K)=nT(A)(K)\;.$$

Let us give some details on it.

Let $\alpha$ has infinitely many lifts, then consider the exact sequence

$$0\to T(A)(K)\to A_0(A)(K)\to A(K)$$

which can be achieved

$$0\to T(A)\to A_0(A)\to A\to 0$$

by taking Galois invariants.

By Roitman's torsion theorem, the $H^1(G,T(A))=0$, $G$ is the Galois group, $T(A)$ is uniquely divisible (applying the Roitman torsion theorem).
So we have 
$$0\to T(A)(K)\to A_0(A)(K)\to A(K)\to 0$$

Therefore, the lift of $\alpha$ can be mapped to zero or non-zero by the albanese map.

If the image of $alb_A(\widetilde{\alpha_i})\neq 0$ is in $A(K)$, which is finitely generated by the Mordell-Weil theorem for abelian varieties. Therefore, $\widetilde{\alpha_i}$, such that its image under the albanese map is non-zero gives finitely many $\ZZ$-linearly independent elements in $A_0(A)(K)$.

So suppose that there are infinitely many $\ZZ$-independent 
$\widetilde{\alpha_i}$ such that it is in $T(A)(K)$.

Therefore, there exist $\beta_i$ infinitely many such that

$$n\beta_i=\widetilde{\alpha_i}$$
in $A_0(A)(K)$.

This contradicts the fact that $\alpha_i\neq 0$ in the quotient group

$$A_0(A)(K)/nA_0(A)(K)$$ 

This means that, given every element $\alpha$ in 

$$A_0(A)(K)/nA_0(A)(K)$$ 

there are finitely many $\ZZ$-independent lifts of $\alpha$, say 
$$\widetilde{\alpha_1}, \cdots, \widetilde{\alpha_k}\;.$$

Since $A(K)$ is finitely generated hence we are remaining to prove that $T(A)(K)$ is finitely generated. Suppose not, then there are infinitely many $\ZZ$-independent $\widetilde{\alpha_i}$ in $T(A)(K)$. For each $i$, there exists a unique $\beta_i$, such that 
$$\widetilde{\alpha_i}=n\beta_i$$
Since $T(A)(K)\subset A_0(A)(K)$ it follows that this equality occurs in $A_0(A)(K)$ as well. So, the class of $\widetilde{\alpha_i}$ under the quotient map $A_0(A)(K)\to A_0(A)(K)/n$ is zero. So, the kernel  of the quotient map is infinitely generated. Also, observe that the kernel contains $T(A)(K)$. So $T(A)(K)\subset nA_0(A)(K)$. 

Suppose that there is $\gamma\in nA_0(A)(K)$ and not in $T(A)(K)$, then 
$$alb_A(\gamma)\neq 0$$
in $A(K)$. So we have to prove that $nA_0(A)(K)/T(A)(K)$ is finitely generated.

Consider the exact sequence:
$$0\to nT(A)\to nA_0(A)\to nA\to 0$$
given by $n\alpha\mapsto n alb_A(\alpha)$, this map is surjective as $A_0(A),A$ are divisible. Taking the Galois cohomology, we have,
$$0\to nT(A)(K)\to nA_0(A)(K)\to nA(K)\to H^1(G, nT(A))$$
The last term in the exact sequence is zero as $nT(A)$ is also unique $n-divisible$. Also, observe that $T(A)(K)=nT(A)(K)$. So we have $nA_0(A)(K)/T(A)(K)$ that is finitely generated, and hence we are done because of $0\to T(A)(K)\to nA_0(A)(K)\to nA(K)\to 0$. So, given a $\gamma$ in $nA_0(A)(K)$, we have finitely many $\gamma_i$, such that
$$\gamma-\sum_i n_i\gamma_i\in T(A)(K)$$
Now write $\gamma=n\beta, \gamma_i=n\beta_i$.
So we have 
$$nalb_A(\beta-\sum_i n_i\beta_i)=0$$
in $A(K)$. So, by Roitman's theorem we have,
$\beta-\sum_i n_i\beta_i=\alpha$, where $\alpha$ is an n-torsion. Therefore, multiplying by $n$, we have
$$\gamma=\sum_i n_i \gamma_i$$ in $nA_0(A)(K)$. So $nA_0(A)(K)$ is finitely generated and hence we are done.

 This is not all, as the group $A_0(A_K)$ is not always the galois-invariant cycle $A_0(A)(K)$. This is only true if the above two groups are isomorphic. There is a natural map of $A_0(A_K)\to A_0(A(\bar K))$, possibly with a kernel. Let us denote this kernel as $\Sigma$, then we have the exact sequence:
 
$$0\to \Sigma\to A_0(A_K)\to A_0(A(\bar K)) $$

By the classical theory of algebraic cycles, $\Sigma$ is torsion. Hence, by taking Galois invariants we have

$$ 0\to \Sigma(K)\to A_0(A_K)\to A_0(A)(K)$$

Therefore, $\Sigma(K)$ is torsion and we have $A_0(A_K)/\Sigma(K)$ is a finitely generated group, which is a subgroup of $A_0(A)(K)$. Then
tensoring with $\QQ$-coefficients we have $A_0(A_K)\otimes \QQ$ is finitely generated.
\end{proof}

\section{Consequence of the above Theorem}

In this section, we prove the following theorem.
Let $A_0(A_{\QQ})$ denote the group of zero cycles of degree zero of degree zero generated by $\QQ$-rational points on $A$ modulo rational equivalence.

\begin{thm}
Let $A$ be an abelian surface defined over $\QQ$ and suppose that the strong version of the Birch-Swinnerton Dyer conjecture is known for $A$, then the Bloch-Beilinson conjecture for $A_0(A_{\QQ})$ follows from it.
\end{thm}

\begin{proof}
In the last theorem, we have proven that the group $A_0(A_K)\otimes \QQ$ is finitely generated. From this we prove that the albanese kernel over $\bar \QQ$ is trivial. Suppose that it is non-trivial, then take $n\geq 2$. By Roitman's theorem we have $T(A)[n]=0$. So we have 
$$T(A)\to T(A)$$
given by 
$$a\to na$$ an isomorphism. So, the group $T(A)$ is uniquely divisible. 

On the other hand, we have $A_0(A_K)\to A(K)$, the kernel of the albanese map, denoted by $T(A_K)$ a finitely generated group. Now any element in the albanese kernel $T(A)$ is in $T(A_K)$ for some finite extension $K$ of $\QQ$. Now, observe that 

$$T(A_K)=nT(A_K)$$

This is not possible because a finitely generated abelian group cannot be divisible. Therefore, $T(A)=0$ and hence 

$$A_0(A(\bar \QQ))\cong A(\bar \QQ)$$

Moreover, we have $T(A_K)=0$ and hence 

$$A_0(A_K)\subset A(K)\;.$$

Suppose that the rank of $A(\QQ)$ is zero, then it follows that the rank of $A_0(A_{\QQ})$ is zero. Suppose now that the Birch-Swinnerton Dyer conjecture is known for $A_{\QQ}$ and further

$$L(A_{\QQ},1)\neq 0$$

Then we show that, 

$$L(H^3(A),2)\neq 0\;.$$

Observe that we have a non-degenerate pairing 

$$H^1(A,\QQ_l(2))\times H^3(A, \QQ_l)\to \QQ_l$$

given by the Poincare duality for l-adic cohomology theory. So, we have $H^3(A,\QQ_l)$ is dual of $H^1(A,\QQ_l(2))$. Now by the functional equation for the L-function we have 

$$L(H^3(A),s)=L(H^3(A)^{\vee},1-s)=L(H^1(A)(-2),1-s)=L(H^1(A),3-s)$$

Putting $s=2$ we have 

$$L(H^3(A),2)=L(H^1(A),1)\neq 0$$

provide that the latter is known to be non-vanishing.
\end{proof}

\subsection{Kummer surface and zero cycles}

Let $A$ be an abelian surface over $K$, a number field. Consider the involution $i$ on $A$ defined by

$$i(a)=-a$$

Consider the quotient $A/i$, this surface has finitely many quotient singularities. Blowing up these singularities, we have a smooth surface $X$, which is a K3 surface and is called the Kummer K3 surface. Since $X$ can also be obtained as a quotient of the involution acting on $\widetilde{A}$, where $\widetilde{A}$ is the blow-up of $A$  at finitely many fixed points of $i$. The group $CH_0$ is a birational invariant of smooth and projective varieties. We have,

$$A_0(A(\bar K))=A_0(\widetilde{A}(\bar K))$$ 

Since the albanese kernel is trivial for $A$, so is for $\widetilde{A}$. Also, we have 

$$A_0(A(\bar K))^i=A_0(\widetilde{A}(\bar K))^i=A_0(X)$$

with $\QQ$-coefficients. Together, these two facts tell us that the albanese kernel of 

$$A_0(X)\to Alb(X)=0$$ 

is torsion and hence by Roitman's Theorem the albanese kernel is trivial for $X$. 

So we get the following theorem:

\begin{thm}
The group $A_0(X)=0$ for the Kummer K3 surface $X$ associated with $A$.
\end{thm}

\textbf{Acknowledgement: }{The author thanks the hospitality of SRM AP for hosting this project. Thanks to Matt Broe for pointing certain incomplete arguments present in the previous version of the arxiv preprint.}

\end{document}